\documentclass[11pt, a4paper]{article}
\usepackage{amsmath}
\usepackage{fancyhdr}
\usepackage[misc]{ifsym}
\usepackage{color}
\usepackage{amsmath}
\usepackage{amsfonts}
\usepackage{dsfont}
\usepackage{mathrsfs}
\usepackage{bbm}
\usepackage{latexsym}
\usepackage{graphicx}
\usepackage{epstopdf}
\usepackage{amssymb}
\usepackage{indentfirst}
\usepackage{subfigure}
\usepackage{booktabs}
\usepackage{authblk}
\usepackage{cite}
\usepackage{float}
\usepackage[numbers,sort&compress]{natbib}
\usepackage{caption}
\usepackage{booktabs}
\usepackage{threeparttable}
\usepackage{multicol}
\usepackage{multirow}
\usepackage{geometry}
\usepackage{diagbox}
\newtheorem{remark}{Remark}
\newtheorem{theorem}{Theorem}[section]
\newtheorem{lemma}[theorem]{Lemma}

\newtheorem{algorithm}[theorem]{Algorithm}

\numberwithin{equation}{section}

\def\qed{\hfill$\Box$\vspace{8pt}}

\begin{document}

\title{More than one Author with different Affiliations}
\author[1]{Xuyang Na}
\author[1,2]{Xuejun Xu}

\affil[1]{\small School of Mathematical Science, Tongji University, Shanghai 200092, China}
\affil[2]{\small Institute of Computational Mathematics, AMSS, Chinese Academy of Sciences, Beijing 100190, China}
\title{Domain Decomposition Methods for Elliptic Problems with High Contrast Coefficients Revisited}\date{}
\maketitle

{\bf{Abstract}:}
In this paper, we revisit the nonoverlapping domain decomposition methods for solving elliptic problems with high contrast coefficients. Some interesting results are discovered. We find that the Dirichlet-Neumann algorithm and Robin-Robin algorithms may make full use of the ratio of coefficients. Actually, in the case of two subdomains, we show that their convergence rates are $O(\epsilon)$, if $\nu_1\ll\nu_2$, where $\epsilon = \nu_1/\nu_2$ and $\nu_1,\nu_2$ are coefficients of two subdomains. Moreover, in the case of many subdomains, the condition number bounds of Dirichlet-Neumann algorithm and Robin-Robin algorithm are $1+\epsilon(1+\log(H/h))^2$ and $C+\epsilon(1+\log(H/h))^2$, respectively,  where $\epsilon$ may be a very small number in the high contrast coefficients case. Besides, the convergence behaviours of the Neumann-Neumann algorithm and Dirichlet-Dirichlet algorithm may be independent of coefficients while they could not benefit from the discontinuous coefficients. Numerical experiments are preformed to confirm our theoretical findings.

{\bf{Keywords}:} \hspace*{2pt}Diffusion problem,\ \ Discontinuous coefficients,\ \ Finite elements,\ \ Domain decomposition.

\section{Introduction}
Diffusion problem is a quite important model which is encountered in many physical problems and practical application fields. It is of great significance to solve diffusion equations numerically. One of the difficulties is that the diffusion coefficients are usually strongly discontinuous. A natural choice to overcome the difficulty is to use nonoverlapping domain decomposition (DD) methods to solve such kind of problems. Actually, there are lots of literature in the study of solving strong discontinuous problems by nonoverlapping DD methods. For instance, Mandel and Brezina\cite{1996balancing} develop a balancing domain decomposition method for steady-state diffusion problem. In \cite{widlund2001feti}, a FETI algorithm is proposed and it is proved that the bounds on the rate of convergence are independent of possible jumps of the coefficients. In \cite{1994Schwarz,1997Nonstandard}, Sarkis design Schwarz preconditioners for discontinuous coefficients problems by using both conforming and non-conforming elements. In \cite{MR2037041}, a Robin-Robin preconditioner is proposed for advection-diffusion problems with discontinuous coefficients. For more study of this aspect, we refer to \cite{toselli2006domain,mathewddm} and the references cited therein.\\
\indent
We may find that the algorithms in most of the literature achieve convergence rates or condition number bounds independent of the jumps of coefficients. Wether there is a better result in the high contrast coefficients case? The answer is absolutely yes. In this paper, we are interested in investigating how the discontinuous coefficients influence the convergence behaviours of Dirichelt-Neumann (D-N) algorithm, Neumann-Neumann (N-N) algorithm, Dirichlet-Dirichelt (D-D) algorithm and Robin-Robin (R-R) algorithm in the case of two subdomains and the case of many subdomains. We find that the large jumps of coefficients may accelerate the iteration of D-N algorithm and R-R algorithm, which is quite different from many other algorithms. Detailedly, if we suppose $\nu_1, \nu_2$ are the discontinuous coefficients in the case of two subdomains, then the convergence rates of the D-N algorithm and the R-R algorithm will completely depend on the ratio of the smaller coefficient to the larger coefficient, i.e. $\nu_1/\nu_2$ if $\nu_1\ll\nu_2$. Here we should clear that unlike the high contrast coefficients case, the convergence rates of D-N algorithm and R-R algorithm are bounded by a constant which is independent of mesh size $h$ and less than 1 strictly in the case $\nu_1$ equals to $\nu_2$. In the case of many subdomains, the D-N algorithm and the R-R algorithm are always regarded as preconditioned methods and the corresponding condition number bounds are $1+\epsilon(1+\log(H/h))^2$ and $C+\epsilon(1+\log(H/h))^2$, respectively, where $\epsilon$ only depends on the ratio of the high contrast coefficients and the higher the contrast is, the smaller the value of $\epsilon$ is. Gander and Dubois\cite{MR3339213} also find a similar phenomenon in the case of two symmetric subdomains. But they use the Fourier analysis to analyze them, as a result, their result is hard to extend to the case of many subdomains. In this paper, we estimate the convergence rate by analyzing the spectra radium of error reduction operators and analyzing the condition numbers of preconditioned systems. Therefore, our results hold in the case of two subdomains and the case of many subdomains. Besides, we prove that the N-N algorithm and D-D algorithm could never take advantage of the high contrast in discontinuous coefficients unless they deteriorate to D-N algorithm while they may be independent of the jumps of coefficients by choosing suitable weights. Roughly speaking, we can explain the phenomena as follows: the D-N algorithm and R-R algorithm use information of half the subdomains to precondition the whole system and the ratio of coefficients will be reserved in the estimates of convergence rates and condition number bounds; in contrast, energy norms of $\Omega_1$ and $\Omega_2$ need to be controlled by each other in the N-N algorithm and D-D algorithm, therefore, suitable weights are essential to get condition number bounds independent of coefficients. All the results are confirmed by numerical experiments.\\
\indent
The paper is organized as follows: In section 2, we introduce the model problem and four domain decomposition methods. In section 3, we analyze the influence of coefficients on convergence rates in the case of two subdomains with subdomains symmetric and nonsymmetric. In section 4, the preconditioned systems in the case of many subdomains are described and the bounds on the condition numbers are given. Finally, we perform several numerical experiments to verify our conclusions.

\section{Model problems and domain decomposition algorithms}
We consider the following elliptic problem with discontinuous coefficients:
\begin{equation}\label{model_problem}
\begin{cases}
\begin{array}{rll}
-\nabla\cdot(\nu(\mathbf{x})\nabla u) &= f\quad&{\rm in}\ \Omega, \\
u &= 0\quad&{\rm on}\ \partial\Omega, \\
\end{array}
\end{cases}
\end{equation}
where $\Omega$ is a bounded, two-dimensional polygonal domain and the diffusion coefficient $\nu(\mathbf{x})$ is a piecewise constant function
\[
\nu(\mathbf{x}) = \begin{cases}
\nu_1\quad \mathbf{x}\in \Omega_1,\\
\nu_2\quad \mathbf{x}\in \Omega_2.\\
\end{cases}
\]
Here $\Omega_1,\Omega_2$ are nonoverlapping subdomains which form a decomposition of $\Omega$ and $\Gamma$ denotes their common interface, i.e. $\Gamma = \partial\Omega_1\cap\partial\Omega_2$.\\
\indent
Let $\mathcal{T}_h$ be a quasi-uniform and regular triangulation of $\Omega$ with the mesh size $h$ and assume that interface $\Gamma$ does not cut through any elements of $\mathcal{T}_h$. Let $W\subset H_0^1(\Omega)$ be a P1 conforming finite element space over $\mathcal{T}_h$. Besides, we need the following finite element spaces,
\begin{equation*}
W_i = W\cap H^1(\Omega_i),\quad W_i^0 = W\cap H^1_0(\Omega_i),\quad i = 1,2,
\end{equation*}
and the space of the interface $\Gamma$,
\begin{equation*}
V_{\Gamma} = W|_{\Gamma}.
\end{equation*}
Then, the weak form of (\ref{model_problem}) is as follows: Find $u\in W$, such that
\[
a(u,v) = (f,v)\quad \forall v\in W,
\]
where
\[
a_i(u,v) = \int_{\Omega_i}\nu_i\nabla u\cdot\nabla v \quad\forall u,v\in W_i, i=1,2,
\]
\[
(f,v)_i = \int_{\Omega_i}fv \quad\forall v\in W_i, i=1,2,
\]
and
\[
a(u,v) = a_1(u,v)+a_2(u,v),\quad (f,v) = (f,v)_1+(f,v)_2.
\]
We also use the following bilinear form on the interface,
\[
\langle u,v\rangle = \int_{\Gamma}uv\quad\forall u,v\in V_{\Gamma}.
\]
\indent
The model problem may be written equivalently in the following multidomain formulation:
\[
\begin{cases}
\begin{array}{rll}
-\nu_1\Delta u &= f\quad &{\rm in\ }\Omega_1,\\
u_1 &= u_2,\quad &{\rm on\ }\Gamma,\\
\nu_1\dfrac{\partial u_1}{\partial\mathbf{n}_1}&= -\nu_2\dfrac{\partial u_2}{\partial\mathbf{n}_2} &{\rm on\ }\Gamma,\\
-\nu_2\Delta u &= f\quad &{\rm in\ }\Omega_2.\\
\end{array}
\end{cases}
\]
The second and the third equations corresponding to Dirichlet and Neumann boundary conditions are imposed to ensure the continuity of the solution and the flux across the interface $\Gamma$. To solve the multidomain problem, we have the following three iterative methods and we would like to write them into weak forms.

\begin{algorithm}[The Dirichlet-Neumann Algorithm\cite{toselli2006domain}]\label{DNalgorithm}
Given $u_{\Gamma}^0 (=0)\in V_{\Gamma}$, compute as the following steps until converge,\\
\textbf{Step 1} solve the Dirichlet problem in $\Omega_1$,
\begin{equation*}
\begin{cases}
\begin{array}{rll}
a_1(u_1^{n+1},v) &= (f,v)_1 \quad &\forall v\in W_1^0,\\
u_1^{n+1} &= 0\quad &{\rm on\ }\partial\Omega_1\backslash\Gamma,\\
u_1^{n+1} &= u_{\Gamma}^{n}\quad &{\rm on\ }\Gamma,\\
\end{array}
\end{cases}
\end{equation*}
\textbf{Step 2} solve the Neumann problem in $\Omega_2$,
\begin{align*}
a_2(u_2^{n+1},v) &=(f,v)_2-\langle\nu_1\frac{\partial u_1^{n+1}}{\partial\mathbf{n}_1},v\rangle\\
 &= (f,v)_2+(f,T_1\gamma_0v)_1-a_1(u_1^{n+1},T_1\gamma_0v)\quad \forall v\in W_2,
\end{align*}
where $T_i: V_{\Gamma}\rightarrow W_i$ is an arbitrary extension operator,\\
\textbf{Step 3} get the next iterate by a relaxation,\\
\[
u_{\Gamma}^{n+1} = \theta u_2^{n+1}+(1-\theta)u_{\Gamma}^n\quad {\rm on\ }\Gamma
\]\\
with an appropriate $\theta$.
\end{algorithm}
\begin{algorithm}[The Neumann-Neumann Algorithm\cite{toselli2006domain}]\label{NNalgorithm}
Given $u_{\Gamma}^0 (=0)\in V_{\Gamma}$, compute as the following steps until converge,\\
\textbf{Step 1} solve the Dirichlet problems in $\Omega_i, i =1,2,$
\[
\begin{cases}
\begin{array}{rll}
a_i(u_i^{n+1},v) &= (f,v)_i \quad &\forall v\in W_i^0,\\
u_i^{n+1} &= 0\quad &{\rm on\ }\partial\Omega_i\backslash\Gamma,\\
u_i^{n+1} &= u_{\Gamma}^{n}\quad &{\rm on\ }\Gamma,\\
\end{array}
\end{cases}
\]
\textbf{Step 2} solve the Neumann problems in $\Omega_i, i =1,2,$
\begin{align*}
a_i(w_i^{n+1},v) &=\delta_i^{\dag}\langle\nu_1\frac{\partial u_1^{n+1}}{\partial\mathbf{n}_1}+\nu_2\frac{\partial u_2^{n+1}}{\partial\mathbf{n}_2},v\rangle\\
 &= \delta_i^{\dag}\left(a_1(u_1^{n+1},T_1\gamma_0v)-(f,T_1\gamma_0v)_1+a_2(u_2^{n+1},T_2\gamma_0v)-(f,T_2\gamma_0v)_2\right)\quad \forall v\in W_i,
\end{align*}
where $\delta_1^{\dag}$ and $\delta_2^{\dag}$ are positive weights with $\delta_1^{\dag}+\delta_2^{\dag} = 1$,\\
\textbf{Step 3} get the next iterate by a relaxation,
\[
u_{\Gamma}^{n+1} = u_{\Gamma}^{n}-\theta(\delta_1^{\dag}w_1^{n+1}+\delta_2^{\dag}w_2^{n+1}),
\]
with an appropriate $\theta$.
\end{algorithm}
\begin{algorithm}[The Dirichlet-Dirichlet Algorithm\cite{toselli2006domain}]\label{DDalgorithm}
Given $\lambda_{\Gamma}^0 (=0)\in V_{\Gamma}$, compute as the following steps until converge,\\
\textbf{Step 1} set $\lambda_1^n = -\lambda_2^n = \lambda_{\Gamma}^n$, solve the Neumann problems with in $\Omega_i, i = 1,2,$
\[
a_i(u_i^{n+1},v) = (f,v)_i+\langle \lambda_i^n,v\rangle\quad\forall v\in H_{\Gamma}^1(\Omega_i),
\]
\textbf{Step 2} solve the Dirichlet problem in $\Omega_i, i=1,2$,
\[
\begin{cases}
\begin{array}{rll}
a_i(w_i^{n+1},v) &= 0\quad &\forall v\in W_i^0,\\
w_i^{n+1} &= 0 \quad &{\rm on\ }\partial\Omega_i\backslash\Gamma,\\
w_i^{n+1} &= \delta_i^{\dag}(u_1^{n+1}-u_2^{n+1})\quad &{\rm on\ }\Gamma,
\end{array}
\end{cases}
\]
where $\delta_1^{\dag}$ and $\delta_2^{\dag}$ are positive weights with $\delta_1^{\dag}+\delta_2^{\dag} = 1$,\\
\textbf{Step 3} get the next iterate by a relaxation,
\[
\lambda_{\Gamma}^{n+1} = \lambda_{\Gamma}^n-\theta(\delta_1^{\dag}\nu_1\frac{\partial w_1^{n+1}}{\partial\mathbf{n}_1}+\delta_2^{\dag}\nu_2\frac{\partial w_2^{n+1}}{\partial\mathbf{n}_2}),
\]
with an appropriate $\theta$.
\end{algorithm}

The matching conditions may be changed equivalently by the combinations of the Dirichlet and Neumann interface conditions as follows:
\[
\begin{cases}
\gamma_1u_1+\nu_1\dfrac{\partial u_1}{\partial\mathbf{n}_1} = \gamma_1u_2+\nu_2\dfrac{\partial u_2}{\partial\mathbf{n}_1} =:g_1\quad{\rm on\ }\Gamma\\
\gamma_2u_2+\nu_2\dfrac{\partial u_2}{\partial\mathbf{n}_2} = \gamma_2u_1+\nu_1\dfrac{\partial u_1}{\partial\mathbf{n}_2} =:g_2\quad{\rm on\ }\Gamma\\
\end{cases}
\]
where the Robin parameters $\gamma_1,\gamma_2$ are positive numbers. Therefore, we have the following Robin-Robin algorithm.
\begin{algorithm}[The Robin-Robin Algorithm\cite{chen2014optimal}]\label{RRalgorithm}
Given $g_1^0 (=0)\in V_{\Gamma}, \gamma_1,\gamma_2 > 0$, compute as the following steps until converge,\\
\textbf{Step 1} solve the problem with Robin boundary condition in $\Omega_1$,
\[
a_1(u_1^n,v)+\gamma_1\langle u_1^n,v\rangle = (f,v)_1+\langle g_1^n,v\rangle\quad\forall v\in W_1,
\]
\textbf{Step 2} update the interface condition,
\[
g_2^n = (\gamma_1+\gamma_2)u_1^n-g_1^n\quad{\rm on\ }\Gamma,
\]
\textbf{Step 3} solve the problem with Robin boundary condition in $\Omega_2$,
\[
a_2(u_2^n,v)+\gamma_2\langle u_2^n,v\rangle = (f,v)_2+\langle g_2^n,v\rangle\quad\forall v\in W_2,
\]
\textbf{Step 4} update the interface condition,
\[
\tilde{g}_1^n = (\gamma_1+\gamma_2)u_2^n-g_2^n\quad{\rm on\ }\Gamma,
\]
\textbf{Step 5} get the next iterate by a relaxation,
\[
g_1^{n+1} = \theta\tilde{g}_1^n+(1-\theta)g_1^n,
\]
with an appropriate $\theta$.
\end{algorithm}

\section{Influence of discontinuous coefficients on convergence rates}
In this section, we will explore the influence of discontinuous coefficients on convergence rates and confirm the optimal parameters of the algorithms in the previous section.\\
\indent
First, we give some preliminaries. Define $\mathcal{H}_i: V_{\Gamma}\rightarrow W_i$ as follows:
\[
\begin{cases}
\begin{array}{rll}
a_i(\mathcal{H}_iu_{\Gamma},v) &= 0\quad&\forall v\in W_i^0,\\
\mathcal{H}_iu_{\Gamma} &= 0 \quad&{\rm on\ }\partial\Omega_i\backslash\Gamma,\\
\mathcal{H}_iu_{\Gamma} &= u_{\Gamma}\quad &{\rm on\ }\Gamma,
\end{array}
\end{cases}
\]
The operator $\mathcal{H}_i$ is known as the `discrete harmonic extension'. We note that the coefficient $\nu_i$ in $a_i(\cdot,\cdot)$ can be omitted because of the zero source term. Define $S_i$ to be a linear operator as follows:
\[
\langle S_iu_{\Gamma},v_{\Gamma}\rangle = a_i(\mathcal{H}_iu_{\Gamma},T_iv_{\Gamma})\quad \forall v_{\Gamma}\in V_{\Gamma},
\]
where $T_i:V_{\Gamma}\rightarrow W_i$ is an arbitrary extension operator. Obviously, $S_i$ is symmetric and positive definite.
Then, we will give the error operators of the four DD algorithms in the following lemma. Actually, the proof of the following lemma may be found in \cite{toselli2006domain} and \cite{chen2014optimal}. For completeness, we give a brief proof here.
\begin{lemma}
The error operators of the D-N algorithm, R-R algorithm, D-D algorithm and R-R algorithm are $R_1, R_2, R_3$ and $R_4$, respectively, where
\begin{align*}
R_1 &= I-\theta S_2^{-1}(S_1+S_2),\hfill\\
R_2 &= I-\theta (D_1S_1^{-1}D_1+D_2S_2^{-1}D_2)(S_1+S_2),\hfill\\
R_3 &= I-\theta(D_1S_1 D_1+D_2S_2 D_2)(S_1^{-1}+S_2^{-1})\hfill
\end{align*}
and
\begin{align*}
R_4 &= I-\theta(\gamma_1I-S_2)(\gamma_2I+S_2)^{-1}((\gamma_2I+S_2)(\gamma_1I-S_2)^{-1}-(\gamma_2I-S_1)(\gamma_1I+S_1)^{-1}) \\
&=I-\theta(I-(\gamma_1I-S_2)(\gamma_2I+S_2)^{-1}(\gamma_2I-S_1)(\gamma_1I+S_1)^{-1}).
\end{align*}
\end{lemma}
\begin{proof}
To deduct the error operators, it is sufficient to consider the homogeneous case, $f\equiv 0$, by linearity. For simplicity, we use the same letters to denote the functions and corresponding errors in the proof without causing any confusion.\\
\indent
We first consider the D-N algorithm. From the definition of discrete harmonic extension and Step 1 of Algorithm \ref{DNalgorithm}, we know $u_1^{n+1} = \mathcal{H}_1u_{\Gamma}^n$, then by the definition of $S_i$ and Step 2 of Algorithm \ref{DNalgorithm}, we have
\begin{equation}\label{dnop1}
a_2(u_2^{n+1},T_2\gamma_0v) = -a_1(u_1^{n+1},T_1\gamma_0v) = -\langle S_1u_{\Gamma}^n,v\rangle\quad\forall v\in V_{\Gamma}.
\end{equation}
Here we note that if we set $v = 0$ in (\ref{dnop1}), we have
\begin{equation}\label{dnop3}
a_2(u_2^{n+1},w) = 0\quad\forall w\in W_2^0,
\end{equation}
which reflects that $u_2^{n+1} = \mathcal{H}_2(u_2^{n+1}|_{\Gamma})$. Therefore, it holds that
\begin{equation*}
\langle S_2(u_2^{n+1}|_{\Gamma}),v\rangle = -\langle S_1u_{\Gamma}^n,v\rangle\quad\forall v\in V_{\Gamma}.
\end{equation*}
Because $S_2$ is a positive definite operator, we have
\begin{equation}\label{dnop2}
u_2^{n+1}|_{\Gamma} = -S_2^{-1}S_1u_{\Gamma}^n.
\end{equation}
Combining (\ref{dnop2}) and Step 3 of Algorithm \ref{DNalgorithm}, we obtain the error operator of D-N algorithm, i.e.
\begin{align*}
u_{\Gamma}^{n+1} &= \theta u_2^{n+1}|_{\Gamma}+(1-\theta)u_{\Gamma}^n\\
&= -\theta S_2^{-1}S_1u_{\Gamma}^n+(1-\theta)u_{\Gamma}^n\\
&= (I-\theta S_2^{-1}(S_1+S_2))u_{\Gamma}^n.
\end{align*}
\indent
We next consider the N-N algorithm. From the definition of $S_i$ and Step 1, Step 2 of Algorithm \ref{NNalgorithm}, we have
\begin{equation*}
a_i(w_i^{n+1},T_i\gamma_0v) = \delta_i^{\dag}(\langle S_1u_{\Gamma}^n,v\rangle+\langle S_2u_{\Gamma}^n,v\rangle).
\end{equation*}
Similar to (\ref{dnop3}), we have $w_i^{n+1} = \mathcal{H}_i(w_i^{n+1}|_{\Gamma})$. Therefore, it holds that
\begin{equation}\label{nnop1}
w_i^{n+1}|_{\Gamma} = S_i^{-1}(\delta_i^{\dag}(S_1+S_2)u_{\Gamma}^n).
\end{equation}
By (\ref{nnop1}) and Step 3 of Algorithm \ref{NNalgorithm}, the error operator is obtained as follows:
\begin{align*}
u_{\Gamma}^{n+1} &= u_{\Gamma}^n-\theta(\delta_1^{\dag}S_1^{-1}(\delta_1^{\dag}(S_1+S_2)u_{\Gamma}^n)+\delta_2^{\dag}S_2^{-1}(\delta_2^{\dag}(S_1+S_2)u_{\Gamma}^n))\\
&= \left(I-\theta(D_1S_1^{-1}D_1+D_2S_2^{-1}D_2)(S_1+S_2)\right)u_{\Gamma}^n,
\end{align*}
where $D_i = \delta_i^{\dag}I, i=1,2.$\\
\indent
The error operator of D-D algorithm may be derived similar to the N-N algorithm. By the definitions of $\mathcal{H}_i,S_i$ and Step 1,2 of Algorithm \ref{DDalgorithm}, we have
\begin{align*}
w_i^{n+1} &= \mathcal{H}_i(\delta_i^{\dag}(u_1^{n+1}|_{\Gamma}-u_2^{n+1}|_{\Gamma}))\\
&= \mathcal{H}_i(\delta_i^{\dag}(S_1^{-1}\lambda_1^n-S_2^{-1}\lambda_2^n))\\
&= \mathcal{H}_i(\delta_i^{\dag}(S_1^{-1}+S_2^{-1})\lambda_{\Gamma}^n).
\end{align*}
Then by the interface update condition in Step 3, we get
\begin{equation*}
\lambda_{\Gamma}^{n+1} = \left(I-\theta(D_1S_1D_1+D_2S_2D_2)(S_1^{-1}+S_2^{-1})\right)\lambda_{\Gamma}^n.
\end{equation*}
\indent
As to the R-R algorithm, for $i = 1,2$, we have the following error equation,
\begin{equation*}
a_i(u_i^n,T_i\gamma_0v)+\gamma_i\langle u_i^n|_{\Gamma},v\rangle = \langle g_i^n,v\rangle\quad\forall v\in V_{\Gamma}.
\end{equation*}
By the definitions of $\mathcal{H}_i, S_i, i = 1,2$, we have
\begin{equation}\label{rrop1}
g_i^n = (\gamma_iI+S_i)u_i^n|_{\Gamma}.
\end{equation}
Using the interface update in Step 2 of Algorithm \ref{RRalgorithm} and (\ref{rrop1}), we have
\begin{align}\label{rrop2}
g_2^n &= (\gamma_1+\gamma_2)u_1^n|_{\Gamma}-g_1^n\nonumber\\
&= (\gamma_1+\gamma_2)(\gamma_1I+S_1)^{-1}g_1^n-(\gamma_1I+S_1)(\gamma_1I+S_1)^{-1}g_1^n\nonumber\\
&= (\gamma_2I-S_1)(\gamma_1I+S_1)^{-1}g_1^n.
\end{align}
Then by the second interface update in Step 4 of Algorithm \ref{RRalgorithm}, (\ref{rrop2}) and (\ref{rrop1}), it holds that
\begin{align}\label{rrop3}
\tilde{g}_1^n &= (\gamma_1+\gamma_2)u_2^n-g_2^n\nonumber\\
&= (\gamma_1+\gamma_2)(\gamma_2I+S_2)^{-1}g_2^n-(\gamma_2I+S_2)(\gamma_2I+S_2)^{-1}g_2^n\nonumber\\
&= (\gamma_1I-S_2)(\gamma_2I+S_2)^{-1}g_2^n\nonumber\\
&= (\gamma_1I-S_2)(\gamma_2I+S_2)^{-1}(\gamma_2I-S_1)(\gamma_1I+S_1)^{-1}g_1^n.
\end{align}
At last, by the relaxation step and (\ref{rrop3}), we obtain the error operator of R-R algorithm, i.e.
\begin{align*}
g_1^{n+1} &= \theta\tilde{g}_1^n+(1-\theta)g_1^n\\
&= \left(I-\theta(I-(\gamma_1I-S_2)(\gamma_2I+S_2)^{-1}(\gamma_2I-S_1)(\gamma_1I+S_1)^{-1})\right)g_1^n\\
&= \left(I-\theta(\gamma_1I-S_2)(\gamma_2I+S_2)^{-1}((\gamma_2I+S_2)(\gamma_1I-S_2)^{-1}-(\gamma_2I-S_1)(\gamma_1I+S_1)^{-1})\right)g_1^n.
\end{align*}
\qed
\end{proof}\\
\indent
The next task is to analyze how the spectral radius of $R_i$ rely on the discontinuous coefficients $\nu_1,\nu_2$ and how to choose appropriate $\theta$ and $\delta_i^{\dag}$ to accelerate the iterations.
\subsection{Symmetric case}
In this subsection, we suppose that $\Omega_1$ and $\Omega_2$ are symmetric with respect to $\Gamma$.
\begin{figure}[h]
\centering
\includegraphics[width=180mm]{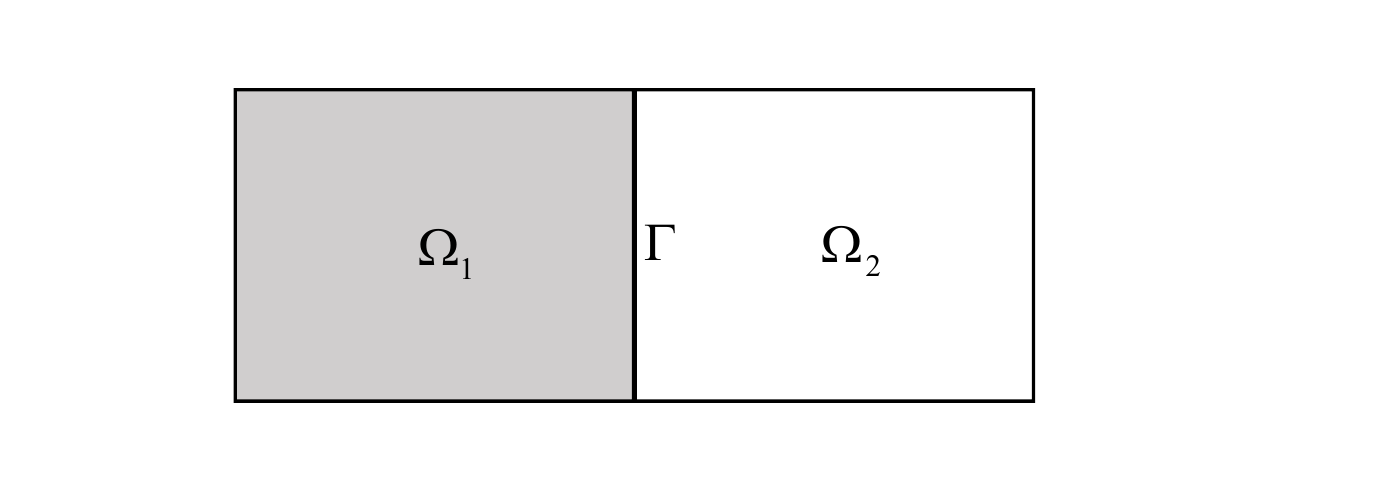}
\caption{$\Omega$ is divided into two symmetric subdomains $\Omega_1,\Omega_2$ and their interface $\Gamma$.}
\label{redblack}
\end{figure}
Then we have
\begin{equation}\label{symmeq}
S_1/\nu_1 = S_2/\nu_2.
\end{equation}
Therefore, $S_1$ and $S_2$ have the same eigenvectors. Let $\nu_1\lambda$ and $\nu_2\lambda$ be the eigenvalues of $S_1$ and $S_2$ corresponding to eigenvector $v$. Then it is easy to check that
\begin{align*}
&\lambda(R_1) = 1-\theta(1+\dfrac{\nu_1}{\nu_2}),\\
&\lambda(R_2) = 1-\theta((\delta_1^{\dag})^2(1+\dfrac{\nu_2}{\nu_1})+(\delta_2^{\dag})^2(1+\dfrac{\nu_1}{\nu_2})),\\
&\lambda(R_3) = 1-\theta((\delta_1^{\dag})^2(1+\dfrac{\nu_1}{\nu_2})+(\delta_2^{\dag})^2(1+\dfrac{\nu_2}{\nu_1})),\\
&\lambda(R_4) = 1-\theta(1-\frac{\gamma_1-\nu_2\lambda}{\gamma_1+\nu_1\lambda}\cdot\frac{\gamma_2-\nu_1\lambda}{\gamma_2+\nu_2\lambda}),
\end{align*}
where $\lambda(T)$ denotes the eigenvalue of operator $T$.\\
\indent
We may find that the eigenvalues of $R_1,R_2,R_3$ are independent of $\lambda$, thus we have
\begin{theorem}\label{0convergence}
In the symmetric case, the convergence rate of algorithm \ref{DNalgorithm}, \ref{NNalgorithm}, \ref{DDalgorithm} can be reduced to $0$ by choosing suitable $\theta$, that is to say, the methods become direct solvers.
\end{theorem}
\begin{proof}
Set
\begin{align*}
&\theta_1 = \frac{1}{1+\dfrac{\nu_1}{\nu_2}},\\
&\theta_2 = \frac{1}{(\delta_1^{\dag})^2(1+\dfrac{\nu_2}{\nu_1})+(\delta_2^{\dag})^2(1+\dfrac{\nu_1}{\nu_2})},\\
&\theta_3 = \frac{1}{(\delta_1^{\dag})^2(1+\dfrac{\nu_1}{\nu_2})+(\delta_2^{\dag})^2(1+\dfrac{\nu_2}{\nu_1})},
\end{align*}
and we get the conclusion.\qed
\end{proof}
\begin{remark}
For a general $\theta$, the convergence behaviours will be quite different between D-N algorithm and N-N algorithm, D-D algorithm.\\
\indent
For a $\theta$ near $1$, the convergence rate of D-N algorithm relies on $\nu_1/\nu_2$ and if $\nu_1\ll\nu_2$, the iteration will perform quite well. In other word, the D-N algorithm benefit from the jump of discontinuous coefficients. But for N-N algorithm and D-D algorithm, we could not obtain such a good property. The range of function
\[
f(\delta_1^{\dag}) = (\delta_1^{\dag})^2(1+\nu_2/\nu_1)+(1-\delta_1^{\dag})^2(1+\nu_1/\nu_2)
\]
is $\left[1,\max\{1+\nu_1/\nu_2,1+\nu_2/\nu_1\}\right]$. So N-N algorithm could not benefit from the discontinuous coefficients. An optimal choice of $\delta_1^{\dag}, \delta_2^{\dag}$ is
\[
\delta_1^{\dag} = \frac{\sqrt{\nu_1}}{\sqrt{\nu_1}+\sqrt{\nu_2}},\quad \delta_2^{\dag} = \frac{\sqrt{\nu_2}}{\sqrt{\nu_1}+\sqrt{\nu_2}},
\]
then
\[
f(\delta_1^{\dag}) = \frac{2(\nu_1+\nu_2)}{(\sqrt{\nu_1}+\sqrt{\nu_2})^2}\in (1,2)
\]
and the convergence rate of N-N algorithm will be independent of $\nu_1, \nu_2$. The case of D-D algorithm is similar.
\end{remark}

\indent
For the Robin-Robin algorithm, we may find that the spectrum depends on the eigenvalue $\lambda$, which is different from other three algorithms. The following theorem can be obtained by analyzing the function $\lambda(R_4)$.
\begin{theorem}
In the symmetric case, the convergence rate of algorithm \ref{RRalgorithm} is bounded by $C\nu_1/\nu_2$ by choosing $\gamma_1 \ge C_1\nu_2 h^{-1}$, $0 < \gamma_2 \le c_0\nu_1$ and $\theta = \frac{2}{2+\nu_1/\nu_2}$.
\end{theorem}
\begin{proof}
Let
\[\lambda(R_4) = 1-\theta(1+\omega(\lambda))
\]
where
\[ \omega(\lambda) = -\frac{\gamma_1-\nu_2\lambda}{\gamma_1+\nu_1\lambda}\cdot\frac{\gamma_2-\nu_1\lambda}{\gamma_2+\nu_2\lambda}.
\]
The derivation of $\omega(\lambda)$ is
\begin{equation}\label{derivation}
\omega'(\lambda) = \frac{(\gamma_1+\gamma_2)(\nu_1+\nu_2)(\gamma_1\gamma_2-\nu_1\nu_2\lambda^2)}{(\gamma_1+\nu_1\lambda)^2(\gamma_2+\nu_2\lambda)^2}.
\end{equation}
It is known \cite{chen2014optimal,xu2010spectral} that
\[
\lambda\in\left[c_0,C_1h^{-1}\right],
\]
then by choosing $0 < \gamma_2 \le c_0\nu_1$, $\gamma_1 \ge C_1\nu_2 h^{-1}$ and (\ref{derivation}), $\omega(\lambda)$ attains the maximum value at $\lambda_0 = \sqrt{\dfrac{\gamma_1\gamma_2}{\nu_1\nu_2}}$ and we have
\[
\omega(\lambda)\in(0,t^2(\dfrac{\eta-1/t}{\eta+t})^2]\subset(0,t^2),
\]
where $\eta = \sqrt{\dfrac{\gamma_1}{\gamma_2}}, t = \sqrt{\dfrac{\nu_1}{\nu_2}}$. Then the optimal choice of $\theta$ is obtained by
\[
1-\theta+1-\theta(1+\dfrac{\nu_1}{\nu_2}) = 0.
\]
That is $\theta = \theta_0 = \frac{2}{2+\nu_1/\nu_2}$, and the convergence rate is bounded by $1-\theta_0(1+\dfrac{\nu_1}{\nu_2}) < \dfrac{\nu_1}{2\nu_2}$.
\qed
\end{proof}\\
\indent
We may find that the Robin-Robin algorithm benefits from the jump of discontinuous coefficients in the symmetric case.
\subsection{Nonsymmetric case}
\begin{figure}[h]
\centering
\includegraphics[width=180mm]{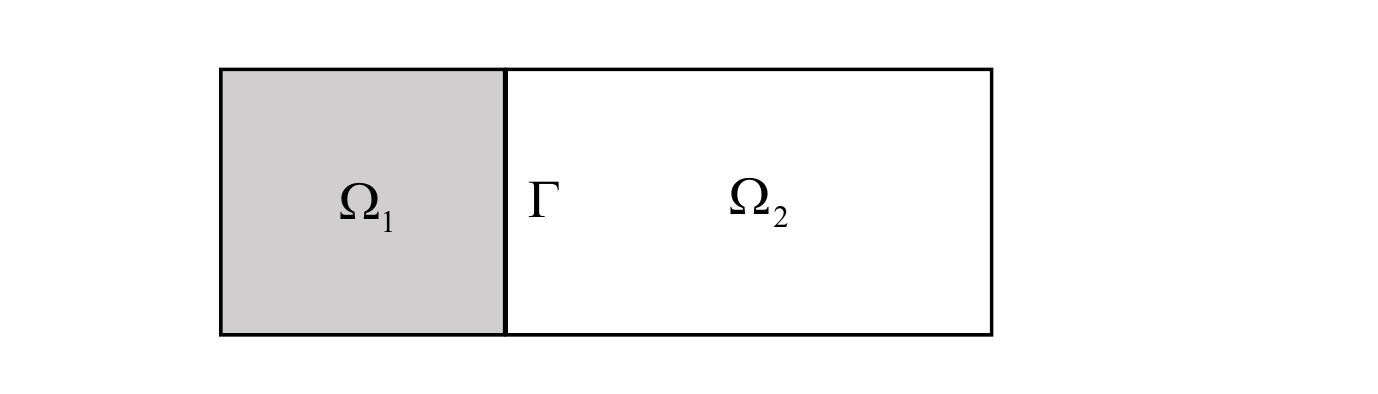}
\caption{$\Omega$ is divided into two nonsymmetric subdomains $\Omega_1,\Omega_2$ and their interface $\Gamma$.}
\label{redblack}
\end{figure}
In the nonsymmetric case, the equality (\ref{symmeq}) is no longer available. Instead, the next lemma is useful in the analysis.
\begin{lemma}
There exists positive constants $c_i, C_i$, independent of $h$, such that for any $v\in W_{\Gamma}$,
\begin{equation}\label{normequiv}
c_i\vert v\vert_{1,\Omega_i}^2\le \Vert v\Vert_{H_{00}^{1/2}(\Gamma)}^2\le C_i\vert v\vert_{1,\Omega_i}^2.
\end{equation}
\end{lemma}

\indent
By (\ref{normequiv}), we may get the equivalence between $S_1$ and $S_2$ while they no longer have the same eigenvector. Therefore, none of the algorithms is a direct solver. Actually, the algorithms could be divided into two groups according to their convergence behaviours. The first group contains D-N algorithm and R-R algorithm. Both of them could benefit from the jump of discontinuous coefficients. To be detailed, we have the following two results.
\begin{theorem}
In the discontinuous coefficients case, the convergence rate of D-N algorithm will be bounded by $\dfrac{\nu_1}{\nu_2}$ if $\nu_1\ll\nu_2$.
\end{theorem}
\begin{proof}
We know
\[
R_1 = 1-\theta S_2^{-1}S = 1-\theta S_2^{-1}(S_1+S_2),
\]
and we just need to find the spectrum of $S_2^{-1}S$.\\
\indent
By (\ref{normequiv}), for any $u_{\Gamma}\in W_{\Gamma}$,
\begin{equation}\label{dn1}
\lambda(S_2^{-1}S) = \frac{\langle Su_{\Gamma},u_{\Gamma}\rangle}{\langle S_2u_{\Gamma},u_{\Gamma}\rangle},
\end{equation}
\begin{equation}\label{dn2}
\langle S_2u_{\Gamma},u_{\Gamma}\rangle\le\langle Su_{\Gamma},u_{\Gamma}\rangle\le\langle (1+\frac{C_2}{c_1}\cdot\frac{\nu_1}{\nu_2})S_2u_{\Gamma},u_{\Gamma}\rangle.
\end{equation}
Combining (\ref{dn1}), (\ref{dn2}), we have
\[
\lambda(S_2^{-1}S)\subset [1,1+\frac{C_2}{c_1}\cdot\frac{\nu_1}{\nu_2}],
\]
and the optimal choice of $\theta$ is $\theta_0 = \frac{2}{2+\frac{C_2}{c_1}\cdot\frac{\nu_1}{\nu_2}}$. Then the convergence rate is bounded by $\dfrac{\nu_1}{\nu_2}$.
\qed
\end{proof}
\begin{theorem}
In the discontinuous coefficients case, if $\nu_1\ll\nu_2$, the convergence rate of the Robin-Robin algorithm will be bounded by $\dfrac{\nu_1}{\nu_2}$ with $\gamma_1 \ge C_0\nu_2 h^{-1}$, $0 < \gamma_2 \le c_0\nu_1$.
\end{theorem}
\begin{proof}
We have
\[
R_4 = I-\theta P_4^{-1}G,
\]
where
\[
P_4^{-1} = (\gamma_1I-S_2)(\gamma_2I+S_2)^{-1},
\]
\[
G = (\gamma_2I+S_2)(\gamma_1I-S_2)^{-1}-(\gamma_2I-S_1)(\gamma_1I+S_1)^{-1}.
\]
For any $g\in W_{\Gamma}$, it holds that\cite{liu2014robin}
\begin{equation}\label{rr1}
\frac{\gamma_2}{\gamma_1}\Vert g\Vert_{0,\Gamma}^2+\frac{\gamma_1+\gamma_2}{\gamma_1^2}\vert g\vert_{S_2}^2\le\langle P_4g,g\rangle\le\frac{\gamma_2}{\gamma_1}\Vert g\Vert_{0,\Gamma}^2+\frac{\gamma_1+\gamma_2}{\gamma_1^2}\vert g\vert_{S_2}^2,
\end{equation}
\begin{equation}\label{rr2}
\frac{\gamma_1+\gamma_2}{(2+\delta)\gamma_1^2}\vert g\vert_{S_1}^2+\frac{\gamma_1+\gamma_2}{\gamma_1^2}\vert g\vert_{S_2}^2\le \langle Gg,g\rangle\le\frac{\gamma_1+\gamma_2}{\gamma_1^2}\vert g\vert_{S_1}^2+\frac{2(\gamma_1+\gamma_2)}{\gamma_1^2}\vert g\vert_{S_2}^2,
\end{equation}
where $\delta$ is an arbitrary positive constant independent of $h$, and $\vert\cdot\vert_{S_i}^2 = \langle S_i\cdot,\cdot\rangle, i =1,2$.
Additionally, we have
\begin{equation}\label{rr4}
\vert g\vert_{S_i}^2\le C\nu_ih^{-1}\Vert g\Vert_{0,\Gamma}^2.
\end{equation}
\indent
By (\ref{rr1}), (\ref{rr2}) and (\ref{rr4}), we have the upper bound estimate, i.e.
\begin{align*}
\langle Gg,g\rangle&\le\frac{\gamma_1+\gamma_2}{\gamma_1^2}\vert g\vert_{S_1}^2+\frac{2(\gamma_1+\gamma_2)}{\gamma_1^2}\vert g\vert_{S_2}^2\\
 &\le \frac{C'(\gamma_1+\gamma_2)}{\gamma_1^2}(1+C\frac{\nu_1}{\nu_2})\vert g\vert_{S_2}^2\\
 &\le (1+C\frac{\nu_1}{\nu_2})\langle P_4g,g\rangle
\end{align*}
with a suitable $\gamma_1\ge C\nu_2 h^{-1}$.\\
\indent
By trace theorem and Poincar$\acute{\text{e}}$ inequality, we have
\begin{equation}\label{rr3}
\Vert g\Vert_{0,\Gamma}^2\le C\vert \mathcal{H}_ig\vert_{1,\Omega_i}^2\le \frac{C}{\nu_i}\vert g\vert_{S_i}^2.
\end{equation}
Then it follows (\ref{rr1}), (\ref{rr2}), (\ref{rr3}) that
\begin{align*}
\langle P_4g,g\rangle &\le \frac{\gamma_2}{\gamma_1}\Vert g\Vert_{0,\Gamma}^2+\frac{\gamma_1+\gamma_2}{\gamma_1^2}\vert g\vert_{S_2}^2\\
 &\le \frac{c_0C}{\gamma_1}\vert g\vert_{S_1}^2+\frac{\gamma_1+\gamma_2}{\gamma_1^2}\vert g\vert_{S_2}^2\\
 &\le \langle Gg,g\rangle
\end{align*}
with a suitable $\gamma_2\le c_0\nu_1$.\\
\indent
Therefore,
\begin{equation}
1\le \lambda(P_4^{-1}G)\le (1+C\frac{\nu_1}{\nu_2}).
\end{equation}
Here $\theta$ is selected to be $\frac{2}{2+C\frac{\nu_1}{\nu_2}}$, and the convergence rate is bounded by $\dfrac{\nu_1}{\nu_2}$.
\qed
\end{proof}
\begin{remark}
In the two theorems above, we assume $\nu_1\ll\nu_2$ and get convergence rates bounded by $\dfrac{\nu_1}{\nu_2}$. If $\nu_2\ll\nu_1$, we could start the iterations of D-N algorithm and R-R algorithm by computing the problem with Dirichlet boundary condition and Robin boundary condition in $\Omega_2$. Then the convergence rates will be bounded by $\dfrac{\nu_2}{\nu_1}$. While, for the case $\nu_1 = \nu_2$, the convergence rates of D-N algorithm and R-R algorithms are bounded by a constant which is independent of $h$ and less than 1 strictly.
\end{remark}

\indent
Unlike the first two algorithms, the N-N algorithm and D-D algorithm could not take advantage of the high contrast coefficients and their common feature is that the convergence rate may be independent of the jump of coefficients with suitable weights.
\begin{theorem}
The convergence rate of the N-N algorithm and D-D algorithm may be independent of $\nu_1, \nu_2$ by choosing suitable $\delta_1^{\dag}, \delta_2^{\dag}$.
\end{theorem}
\begin{proof}
We have
\[
R_2 = I-\theta P_2^{-1}S = I-\theta (D_1S_1^{-1}D_1+D_2S_2^{-1}D_2)(S_1+S_2).
\]
By (\ref{normequiv}), for any $u_{\Gamma}\in W_{\Gamma}$, it holds that
\begin{equation}\label{nn1}
\lambda(P_2^{-1}S) = \frac{\langle Su_{\Gamma},u_{\Gamma}\rangle}{\langle P_2u_{\Gamma},u_{\Gamma}\rangle},
\end{equation}
\begin{equation}\label{nn2}
\frac{c_1}{C_2}\cdot\frac{\nu_2}{\nu_1}\le \frac{\langle S_2u_{\Gamma},u_{\Gamma}\rangle}{\langle S_1u_{\Gamma},u_{\Gamma}\rangle}\le \frac{C_1}{c_2}\cdot\frac{\nu_2}{\nu_1},
\end{equation}
and
\begin{equation}\label{nn3}
\frac{c_2}{C_1}\cdot\frac{\nu_1}{\nu_2}\le \frac{\langle S_1u_{\Gamma},u_{\Gamma}\rangle}{\langle S_2u_{\Gamma},u_{\Gamma}\rangle}\le \frac{C_2}{c_1}\cdot\frac{\nu_1}{\nu_2},
\end{equation}
Combining (\ref{nn1}), (\ref{nn2}), (\ref{nn3}), we have
\begin{equation}\label{rangenn}
(\delta_1^{\dag})^2(1+\frac{c_1}{C_2}\cdot\frac{\nu_2}{\nu_1})+(\delta_2^{\dag})^2(1+\frac{c_2}{C_1}\cdot\frac{\nu_1}{\nu_2})\le \lambda(P_2^{-1}S)\le (\delta_1^{\dag})^2(1+\frac{C_1}{c_2}\cdot\frac{\nu_2}{\nu_1})+(\delta_2^{\dag})^2(1+\frac{C_2}{c_1}\cdot\frac{\nu_1}{\nu_2}).
\end{equation}
We denote the left and right sides of inequality (\ref{rangenn}) by $\lambda_{min}$ and $\lambda_{max}$, then by choosing $\theta = \theta_0 = \frac{2}{\lambda_{min}+\lambda_{max}}$, we have
\[
\lambda(R_2)\subset (-\frac{\lambda_{max}-\lambda_{min}}{\lambda_{max}+\lambda_{min}},\frac{\lambda_{max}-\lambda_{min}}{\lambda_{max}+\lambda_{min}}) = (-\frac{\kappa-1}{\kappa+1},\frac{\kappa-1}{\kappa+1}).
\]
We now analyze how $\kappa = \kappa(\delta_1^{\dag})$ changes with $\delta_1^{\dag}$, where
\[
\kappa(\delta_1^{\dag}) = \frac{\lambda_{max}}{\lambda_{min}} = \frac{(\delta_1^{\dag})^2(1+\frac{c_1}{C_2}\cdot\frac{\nu_2}{\nu_1})+(\delta_2^{\dag})^2(1+\frac{c_2}{C_1}\cdot\frac{\nu_1}{\nu_2})}{(\delta_1^{\dag})^2(1+\frac{C_1}{c_2}\cdot\frac{\nu_2}{\nu_1})+(\delta_2^{\dag})^2(1+\frac{C_2}{c_1}\cdot\frac{\nu_1}{\nu_2})}.
\]
The derivation of $\kappa(\delta_1^{\dag})$ is
\[
\kappa'(\delta_1^{\dag}) = \frac{\left(\frac{\nu_2}{\nu_1}(\frac{C_1}{c_2}-\frac{C_2}{c_1})-\frac{\nu_1}{\nu_2}(\frac{c_1}{C_2}-\frac{c_2}{C_1})\right)\delta_1^{\dag}(1-\delta_1^{\dag})}{\left((\delta_1^{\dag})^2(1+\frac{c_1}{C_2}\cdot\frac{\nu_2}{\nu_1})+(1-\delta_1^{\dag})^2(1+\frac{c_2}{C_1}\cdot\frac{\nu_1}{\nu_2})\right)^2},
\]
then we know the minimum value of $\kappa(\delta_1^{\dag})$ is attained at $\delta_1^{\dag} = 0$ or $\delta_1^{\dag} = 1$ according to the symbol of the coefficient. However, in both the two cases, the N-N algorithm deteriorates to D-N algorithm. Although, we can choose a $\delta_1^{\dag}$ near $1$ or $0$ and it seems that the convergence rate benefits from the discontinuous coefficients, we take it as an asymptotic convergence behaviour and we prefer D-N algorithm. In fact, we may choose
\[\delta_1^{\dag} = \frac{\sqrt{\nu_1}}{\sqrt{\nu_1}+\sqrt{\nu_2}},\quad\delta_2^{\dag} = \frac{\sqrt{\nu_2}}{\sqrt{\nu_1}+\sqrt{\nu_2}},\]
so that
\[
\kappa\le\frac{\max\{\frac{C_1}{c_2},\frac{C_2}{c_1}\}+1}{\min\{\frac{c_2}{C_1},\frac{c_1}{C_2}\}+1},
\]
which is independent of $\nu_1$ and $\nu_2$.\\
\indent
The case of D-D algorithm is similar.
\qed
\end{proof}

\begin{remark}
From the analysis above, we find that the jump of discontinuous coefficients could accelerate the iteration when using the D-N algorithm and the R-R algorithm in the case of two subdomains as the ratio of the smaller coefficient to the larger one ($\nu_1/\nu_2$ when $\nu_1<\nu_2$) dominate their convergence rates. By contrast, the N-N algorithm and the D-D algorithm could not benefit from the ratio because the norms of $\Omega_1, \Omega_2$ need to be controlled by each other.
\end{remark}

\section{The case of many subdomains}
In this section, we further consider how the discontinuous coefficients influence the convergence behaviours of these domain decomposition methods and wether the previous properties still hold in the case of many subdomains.
\begin{figure}[h]
\centering
\includegraphics[width=150mm]{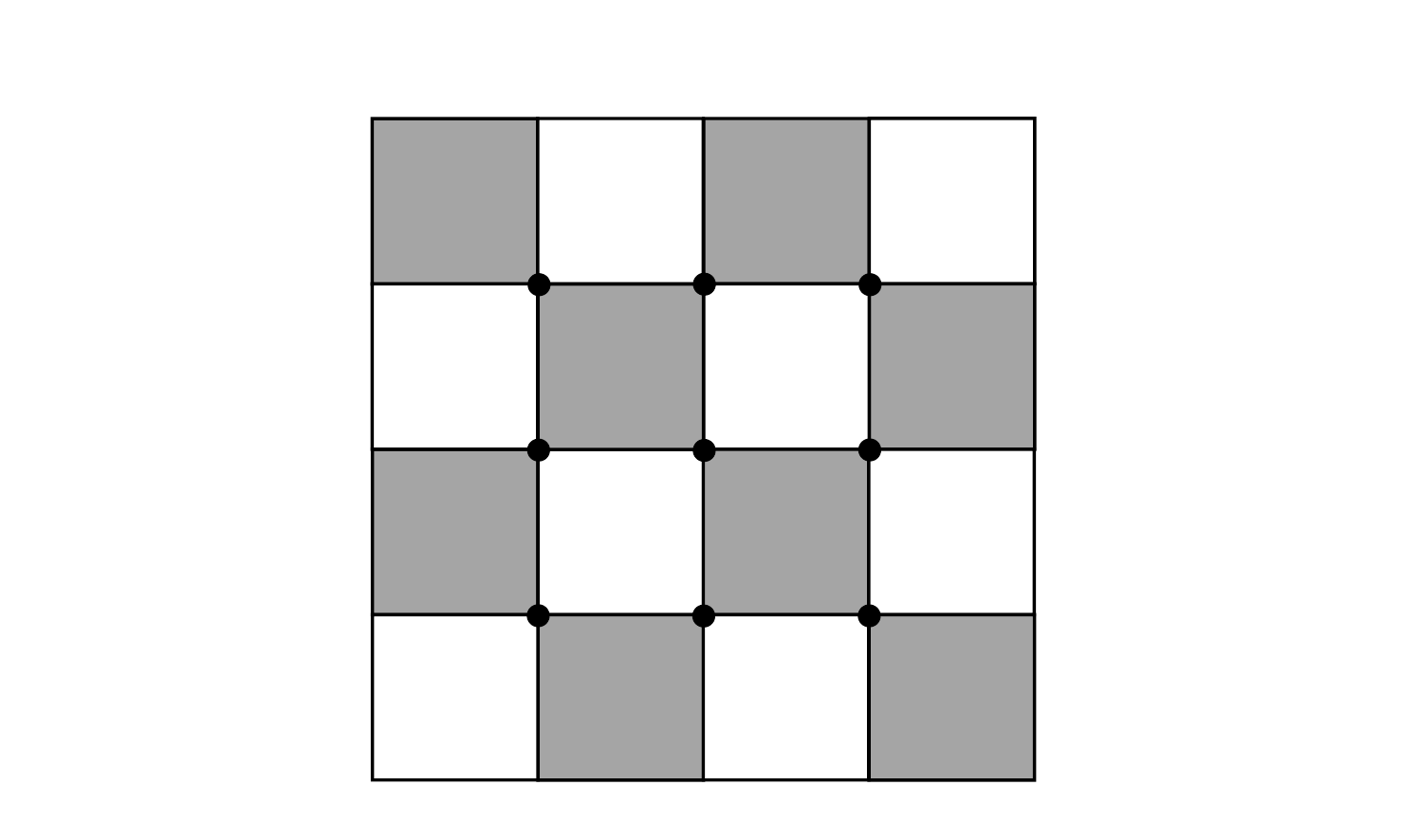}
\caption{The red-black partition of $\Omega$ into $4\times 4$ subdomains. The grey blocks denote the black domains and the white blocks denote the red domains. The black dots are cross points}
\label{redblack}
\end{figure}
\subsection{Preconditioned systems}
The algorithms in the case of many subdomains in this paper rely on a red-black partition, so we first introduce the geometric settings. Partition the domain $\Omega$ into two classes of nonoverlapping subdomains $\Omega_R, \Omega_B$. The red domain is denoted by $\Omega_R = \bigcup \Omega_i, i\in\Lambda_R$ and the black domain is denoted by $\Omega_B = \bigcup\Omega_j, j\in\Lambda_B$, where $\Lambda_R,\Lambda_B$ are the sets of subscripts of subdomains which belong to class $\Omega_R, \Omega_B$, respectively. The size of subdomains is $H$. The intersection of subdomains in the same class is either empty or vertex. The interface is $\Gamma := \partial\Omega_R\cap\partial\Omega_B$. The vertexes of subdomains, which do not belong to $\partial\Omega$, are called the cross points. $\mathcal{T}_h$ is the triangulation same as the case of two subdomains. We assume that the subdomains boundaries do not cut through any element in $\mathcal{T}_h$.\\
\indent
Based on the geometric settings, some function spaces are defined. Let $W\in H_0^1(\Omega)$ be the P1 conforming finite element space. Then let $W_k,k\in\Lambda_R\cap\Lambda_B, W_R, W_B$ be the spaces which include the functions of $W$ restricted to $\overline{\Omega}_k, \overline{\Omega}_R, \overline{\Omega}_B$. $W_k^0, W^0_R, W^0_B$ are the subspaces of $W_k, W_R, W_B$ that functions of $W_k^0, W^0_R, W^0_B$ have vanishing traces on $\partial\Omega_k, \partial\Omega_R, \partial\Omega_B$, respectively. The space on the interface $\Gamma$ is defined to be $V_{\Gamma} = W|_{\Gamma}$. We also denote $V_k := W|_{\partial\Omega_k}$. Besides, $V_{\Delta}$ contains functions in $V_{\Gamma}$ who vanish at each node on cross points.\\
\indent
We then introduce some bilinear forms and operators in the case of many subdomains. Define the bilinear form on the subdomain $\Omega_k, k\in\Lambda_R\cup\Lambda_B$ by
\[
a_k(u_k,v_k) =\int_{\Omega_k}\nu_k\nabla u_k\cdot\nabla v_k\quad\forall u_k,v_k\in W_k.
\]
\indent
Define local discrete harmonic extension operator $\mathcal{H}_k: V_k\rightarrow W_k$ as follows:
\begin{equation*}
\begin{cases}
\begin{array}{rll}
a_k(\mathcal{H}_ku_{\Gamma},v_k) &= 0&\quad\forall v_k\in W_k^0,\\
\mathcal{H}_ku_{\Gamma} &= u_{\Gamma}&\quad{\rm on\ }\Gamma,\\
\end{array}
\end{cases}
\end{equation*}
where $u_{\Gamma}\in V_k$. Here we also note that the constant coefficient $\nu_k$ does not affect the result $\mathcal{H}_ku_{\Gamma}$. The local Schur complement operator $S_k:V_k\rightarrow V_k$ is defined as follows:
\[
\langle S_ku_k, v_k\rangle = a_k(\mathcal{H}_ku_k,T_kv_k)\quad\forall v_k\in V_k,
\]
where $T_k$ is an arbitrary extension operators. We may see that $S_k$ is a symmetric and positive semi-definite operator, therefore it may induce a semi-norm of $V_k$, i.e. $\vert\cdot\vert_{S_k}^2 := \langle S_k\cdot,\cdot\rangle$(if $\partial\Omega_k\cap\partial\Omega\neq\phi$, $S_k$ will be positive definite and it induces a norm).\\
\indent
Based on the local bilinear forms and operators, we could define global operators. Define the bilinear form on $\Omega_R$ by
\begin{equation*}
a_R(u_R,v_R) = \sum_{k\in\Lambda_R}\int_{\Omega_k}\nabla u_R\cdot\nabla v_R\quad\forall u_R,v_R\in W_R.
\end{equation*}
Define discrete harmonic extension operator $\mathcal{H}_R: V_{\Gamma}\rightarrow W_R$ as follows:
\begin{equation*}
\begin{cases}
\begin{array}{rll}
a_R(\mathcal{H}_Ru_{\Gamma},v_R) &= 0&\quad\forall v_R\in W_R^0,\\
\mathcal{H}_Ru_{\Gamma} &= u_{\Gamma}&\quad{\rm on\ }\Gamma.\\
\end{array}
\end{cases}
\end{equation*}
Then the Schur complement operator $S_R:V_{\Gamma}\rightarrow V_{\Gamma}$ could be defined, i.e.
\begin{equation*}
\langle S_Ru_{\Gamma},v_{\Gamma}\rangle = a_R(\mathcal{H}_Ru_{\Gamma},T_Rv_{\Gamma})\quad\forall v_{\Gamma}\in V_{\Gamma},
\end{equation*}
where $T_R:V_{\Gamma}\rightarrow W_R$ is an arbitrary extension operator. From the geometric settings, we may see that $\Omega_R$ consists of $\Omega_k, k\in\Lambda_R$ and they are connected by cross points. Therefore, $S_R$ is a symmetric and positive definite operator and it induces a norm of $V_{\Gamma}$, i.e. $\vert\cdot\vert_{S_R}^2 = \langle S_R\cdot,\cdot\rangle$. Similarly, we may define $a_B(\cdot,\cdot), \mathcal{H}_B$ and $S_B$. The Schur complement operator of the whole subdomains is defined as the sum of $S_R$ and $S_B$, that is, $S = S_R+S_B$.\\
\indent
At last, we give the definitions of Schur complement operators of $V_{\Delta}$. Define $\widetilde{S}: V_{\Delta}\rightarrow V_{\Delta}$ as follows:
\[
\langle \widetilde{S}u_{\Delta},u_{\Delta}\rangle = \min_{u\in V_{\Gamma},u|_{\Gamma_{\Delta}} = u_{\Delta}}\langle Su,u\rangle,
\]
where $\Gamma_{\Delta}$ denotes the degrees of freedom on $\Gamma$ except cross points. From the minimization property and the fact that $S$ is symmetric and positive definite, we know that $\widetilde{S}$ is also a symmetric and positive definite operator and it induces a norm of $V_{\Delta}$. Similarly, $\widetilde{S}_R, \widetilde{S}_B$ could be defined and they hold the same properties as $\widetilde{S}$.\\
\indent
Now we are in a position to introduce algorithms in the case of many subdomains. We use the following Schur complement system in this paper:
\[
\widetilde{S}u_{\Gamma} = \tilde{f}_{\Delta}.
\]
The D-N algorithm and N-N algorithm could provide preconditioners for this system, which are $P_{DN}^{-1} = \widetilde{S}_B^{-1}$ and $P_{NN}^{-1} = D_{R}\widetilde{S}_R^{-1}D_{R}+D_{B}\widetilde{S}_B^{-1}D_{B}$, respectively. Here, $D_{R}, D_{B}$ are scaling operators.\\
\indent
The system of flux is
\[
F\lambda = d,
\]where $F = S_R^{-1}+S_B^{-1}, d = S_B^{-1}f_B-S_R^{-1}f_R$. For the flux system, the D-D algorithm provide a preconditioner $P_{DD}^{-1} = D_{R}S_RD_{R}+D_{B}S_BD_{B}$.
\\
\indent
Similar to the case of two subdomains, the system of Robin boundary data $g_R$ is
\[
Gg_R = ((\gamma_BI+S_B)(\gamma_RI-S_B)^{-1}-(\gamma_BI-S_R)(\gamma_RI+S_R)^{-1})g_R = f^{\star}
\] and the R-R preconditioner is $P_{RR}^{-1} = (\gamma_RI-S_B)(\gamma_BI+S_B)^{-1}$.\\
 \indent
The right hand sides $\tilde{f}_{\Delta},d,f^{\star}$ will be illustrated in detail in the last subsection.
\subsection{Condition number estimate}
In this subsection, we analyze the condition numbers of these preconditioned systems. For simplicity, we consider the red-black checkerboard case, i.e.
\begin{equation*}
\nu(\mathbf{x}) =
\begin{cases}
\nu_R\quad \forall\mathbf{x}\in\Omega_R,\\
\nu_B\quad \forall\mathbf{x}\in\Omega_B,\\
\end{cases}
\end{equation*}
then the scaling operators $D_{R}$ and $D_{B}$ will be $\delta_R^{\dag}I$ and $\delta_B^{\dag}I$, respectively. We also assume that $\nu_R<\nu_B$.\\
\indent
As to D-N algorithms and R-R algorithm, we have following conclusions.
\begin{theorem}
For the D-N algorithm, we have
\begin{equation}\label{thmDN2}
\langle\widetilde{S}_Bu_{\Delta},u_{\Delta}\rangle\le\langle\widetilde{S}u_{\Delta},u_{\Delta}\rangle\le \left(1+C\frac{\nu_R}{\nu_B}(1+\log{\frac Hh})^2\right)\langle\widetilde{S}_Bu_{\Delta},u_{\Delta}\rangle\quad \forall u_{\Delta}\in V_{\Delta}.
\end{equation}
Therefore, the condition number of D-N algorithm is bounded as follows:
\begin{equation}\label{thmDN4}
\kappa(P_{DN}^{-1}\widetilde{S})\le 1+C\frac{\nu_R}{\nu_B}(1+\log{\frac Hh})^2.
\end{equation}
\end{theorem}
\begin{proof}
Suppose
\begin{equation*}
\langle\widetilde{S}u_{\Delta},u_{\Delta}\rangle = \min_{u\in V_{\Gamma},u|_{\Gamma_{\Delta}} = u_{\Delta}}\langle Su,u\rangle = \langle Su_{\Gamma},u_{\Gamma}\rangle = \langle S_Ru_{\Gamma},u_{\Gamma}\rangle+\langle S_Bu_{\Gamma},u_{\Gamma}\rangle,
\end{equation*}
then the lower bound may be obtained by the definition of $\widetilde{S}_B$, i.e.
\begin{equation*}
\langle\widetilde{S}_Bu_{\Delta},u_{\Delta}\rangle = \min_{u\in V_{\Gamma},u|_{\Gamma_{\Delta}} = u_{\Delta}}\langle S_Bu,u\rangle\le\langle S_Bu_{\Gamma},u_{\Gamma}\rangle\le\langle\widetilde{S}u_{\Delta},u_{\Delta}\rangle.
\end{equation*}\\
\indent
On the other hand, suppose
\begin{equation*}
\langle\widetilde{S}_Bu_{\Delta},u_{\Delta}\rangle = \min_{u\in V_{\Gamma},u|_{\Gamma_{\Delta}} = u_{\Delta}}\langle S_Bu,u\rangle = \langle S_B\tilde{u}_{\Gamma},\tilde{u}_{\Gamma}\rangle,
\end{equation*}
then
\begin{equation}\label{manydn1}
\langle\widetilde{S}u_{\Delta},u_{\Delta}\rangle \le \langle S\tilde{u}_{\Gamma},\tilde{u}_{\Gamma}\rangle = \langle S_R\tilde{u}_{\Gamma},\tilde{u}_{\Gamma}\rangle+\langle S_B\tilde{u}_{\Gamma},\tilde{u}_{\Gamma}\rangle,
\end{equation}
Therefore, to prove the upper bound, we need to control $\langle S_R\tilde{u}_{\Gamma},\tilde{u}_{\Gamma}\rangle$ by $\langle S_B\tilde{u}_{\Gamma},\tilde{u}_{\Gamma}\rangle$. Let $\Pi_H: V_{\Gamma}\rightarrow V_{\Gamma}$ be the linear interpolation operator on the coarse grid, where $\Pi_Hv(x) = v(x)$ for any cross point $x$. Then by Lemma 3.1 and Lemma 3.3 in \cite{widlund1988iterative}, we have
\begin{align}\label{manydn2}
\langle S_R\tilde{u}_{\Gamma},\tilde{u}_{\Gamma}\rangle &= \sum_{i\in\Lambda_R}\nu_i\vert\mathcal{H}_i\tilde{u}_{\Gamma}\vert_{1,\Omega_i}^2\nonumber \\
 &\le \sum_{i\in\Lambda_R}\nu_i\left(\vert\mathcal{H}_i(\tilde{u}_{\Gamma}-\Pi_H\tilde{u}_{\Gamma})\vert_{1,\Omega_i}^2+\vert\mathcal{H}_i\Pi_H\tilde{u}_{\Gamma}\vert_{1,\Omega_i}^2\right)\nonumber\\
 &\le C\nu_R\sum_{i\in\Lambda_R}\sum_{j\in\Lambda_i}\left(\Vert(\tilde{u}_{\Gamma}-\Pi_H\tilde{u}_{\Gamma})\Vert_{H^{1/2}_{00}(\Gamma_{ij})}^2+\vert\tilde{u}_{\Gamma}(x_{ij0})-\tilde{u}_{\Gamma}(x_{ij1})\vert^2\right)\nonumber\\
 &\le C\nu_R\sum_{j\in\Lambda_B}\sum_{i\in\Lambda_j}\left(\Vert(\tilde{u}_{\Gamma}-\Pi_H\tilde{u}_{\Gamma})\Vert_{H^{1/2}_{00}(\Gamma_{ji})}^2+\vert\tilde{u}_{\Gamma}(x_{ji0})-\tilde{u}_{\Gamma}(x_{ji1})\vert^2\right)\nonumber\\
 &\le C\nu_R\sum_{j\in\Lambda_B}\nu_j^{-1}(1+\log{\frac Hh})^2\langle S_j\tilde{u}_{\Gamma},\tilde{u}_{\Gamma}\rangle\nonumber\\
 &\le C\frac{\nu_R}{\nu_B}(1+\log{\frac Hh})^2\langle S_B\tilde{u}_{\Gamma},\tilde{u}_{\Gamma}\rangle.
\end{align}\\
\indent
Combining (\ref{manydn1}) and (\ref{manydn2}), we get the upper bound. The estimate of condition number is as follows:
\begin{equation*}
\kappa(P_{DN}^{-1}\widetilde{S}) = \frac{\lambda_{max}(P_{DN}^{-1}\widetilde{S})}{\lambda_{min}(P_{DN}^{-1}\widetilde{S})}\le 1+C\frac{\nu_R}{\nu_B}(1+\log{\frac Hh})^2.
\end{equation*}
\qed
\end{proof}
\begin{theorem}
For the R-R algorithm, we assume $\gamma_R\ge C\nu_B h^{-1}$ and $0<\gamma_B\le c\nu_R H$, then it holds that
\begin{equation}
c\langle P_{RR}g,g\rangle\le\langle Gg,g\rangle \le C\left(1+\frac{\nu_R}{\nu_B}(1+\log{\frac Hh})^2\right)\langle P_{RR}g,g\rangle.
\end{equation}
Thus,
\begin{equation*}
\kappa(P_{RR}^{-1}G)\le C\left(1+\frac{\nu_R}{\nu_B}(1+\log{\frac Hh})^2\right).
\end{equation*}
\end{theorem}
\begin{proof}
For the preconditioned system, we have the following estimates,
\begin{equation}\label{manyrr1}
\frac{\gamma_B}{\gamma_R}\Vert g\Vert_{0,\Gamma}^2+\frac{\gamma_R+\gamma_B}{\gamma_R^2}\vert g\vert_{S_B}^2\le \langle P_{RR}g,g\rangle\le\frac{\gamma_B}{\gamma_R}\Vert g\Vert_{0,\Gamma}^2+ \frac{2(\gamma_R+\gamma_B)}{\gamma_R^2}\vert g\vert_{S_B}^2,
\end{equation}
and
\begin{equation}\label{manyrr2}
\frac{\gamma_R+\gamma_B}{(2+\delta)\gamma_R^2}\vert g\vert_{S_R}^2+ \frac{\gamma_R+\gamma_B}{\gamma_R^2}\vert g\vert_{S_B}^2\le\langle Gg,g\rangle\le \frac{\gamma_R+\gamma_B}{\gamma_R^2}\vert g\vert_{S_R}^2+ \frac{2(\gamma_R+\gamma_B)}{\gamma_R^2}\vert g\vert_{S_B}^2,
\end{equation}
where $\delta > 0$ is an constant independent of $h, H$. For the details, we refer to \cite{liu2014robin}.\\
\indent
By (\ref{manyrr1}) and (\ref{manyrr2}), we may get the upper bound estimate, i.e.
\begin{align}\label{manyrrupper}
\langle Gg,g\rangle &\le \frac{\gamma_R+\gamma_B}{\gamma_R^2}\vert g\vert_{S_R}^2+ \frac{2(\gamma_R+\gamma_B)}{\gamma_R^2}\vert g\vert_{S_B}^2\nonumber\\
 &\le \frac{\gamma_R+\gamma_B}{\gamma_R^2}\left(2+C\frac{\nu_R}{\nu_B}(1+\log{\frac Hh})^2\right)\vert g\vert_{S_B}^2\\
 &\le C\left(1+\frac{\nu_R}{\nu_B}(1+\log{\frac Hh})^2\right)\langle P_{RR}g,g\rangle.
\end{align}\\
\indent
The $L^2$ norm of $g$ on the interface may be estimated by trace theorem, i.e.
\begin{equation*}
\Vert g\Vert_{0,\partial\Omega_i}^2\le CH\vert\mathcal{H}_ig\vert_{1,\Omega_i}^2+CH^{-1}\Vert\mathcal{H}_ig\Vert_{0,\Omega_i}^2\quad i\in\Lambda_R,
\end{equation*}
\begin{equation*}
\Vert g\Vert_{0,\partial\Omega_j}^2\le CH\vert\mathcal{H}_jg\vert_{1,\Omega_j}^2+CH^{-1}\Vert\mathcal{H}_jg\Vert_{0,\Omega_j}^2\quad j\in\Lambda_B.
\end{equation*}
Summing over all the subdomains, we get
\begin{align}\label{manyrrl21}
\Vert g\Vert_{0,\Gamma}^2 &\le CH(\vert\mathcal{H}_Rg\vert_{1,\Omega_R}^2+\vert\mathcal{H}_Bg\vert_{1,\Omega_B}^2) +CH^{-1}(\Vert\mathcal{H}_Rg\Vert_{0,\Omega_R}^2+\Vert\mathcal{H}_Bg\Vert_{0,\Omega_B}^2)\nonumber\\
 &\le CH^{-1}(\vert\mathcal{H}_Rg\vert_{1,\Omega_R}^2+\vert\mathcal{H}_Bg\vert_{1,\Omega_B}^2),
\end{align}
by using Poincar$\acute{\text{e}}$ inequality and the fact that $H\le H^{-1}$.
Then by the choice of $\gamma_R, \gamma_B$ and assumption $\nu_R<\nu_B$, the lower bound could be obtained as follows,
\begin{align}\label{manyrrlower}
\langle P_{RR}g,g\rangle &\le\frac{\gamma_B}{\gamma_R}\Vert g\Vert_{0,\Delta}^2+ \frac{2(\gamma_R+\gamma_B)}{\gamma_R^2}\vert g\vert_{S_B}^2\nonumber\\
&\le C\frac{\nu_R}{\gamma_R}(\vert\mathcal{H}_Rg\vert_{1,\Omega_R}^2+\vert\mathcal{H}_Bg\vert_{1,\Omega_B}^2)+ \frac{2(\gamma_R+\gamma_B)}{\gamma_R^2}\vert g\vert_{S_B}^2\nonumber\\
&\le C\frac{1}{\gamma_R}\vert g\vert_{S_R}^2+(C\frac{\nu_R}{\nu_B\gamma_R}+\frac{2(\gamma_R+\gamma_B)}{\gamma_R^2})\vert g\vert_{S_B}^2\nonumber\\
&\le \max\{C\frac{(2+\delta)\gamma_R}{\gamma_R+\gamma_B}, 2+\frac{\nu_R}{\nu_B}\cdot\frac{\gamma_R}{\gamma_R+\gamma_B}\}\langle Gg,g\rangle\nonumber\\
&\le C\langle Gg,g\rangle.
\end{align}
The condition number is then bounded by using ($\ref{manyrrupper}$) and (\ref{manyrrlower}).
\qed
\end{proof}
\begin{remark}
We may find that D-N algorithm and R-R algorithm could still benefit from the jumps of the discontinuous coefficients. If $\nu_R\ll\nu_B$, the condition numbers will be bounded by a nearly constant, that is,
\begin{equation*}
\kappa(P_{DN}^{-1}\widetilde{S})\le 1+O(\epsilon),
\end{equation*}
and
\begin{equation*}
\kappa(P_{RR}^{-1}G)\le C(1+O(\epsilon)),
\end{equation*}
where $\epsilon = \dfrac{\nu_R}{\nu_B}$.
\end{remark}
\begin{theorem}
For the N-N algorithms, we have
\begin{equation}\label{thmNN2}
\kappa(P_{NN}^{-1}\widetilde{S})\le C(1+\log{\frac Hh})^2
\end{equation}
by choosing suitable $D_{R}$ and $D_{B}$.
\end{theorem}
\begin{proof}
According to the definitions, we have
\begin{align*}
\langle\widetilde{S}u_{\Delta},u_{\Delta}\rangle &= \min_{u\in V_{\Gamma},u|_{\Gamma_{\Delta}} = u_{\Delta}}\langle Su,u\rangle = \langle Su_{\Gamma},u_{\Gamma}\rangle = \langle S_Ru_{\Gamma},u_{\Gamma}\rangle+\langle S_Bu_{\Gamma},u_{\Gamma}\rangle\\
 &\ge\langle\widetilde{S}_Ru_{\Delta},u_{\Delta}\rangle+\langle\widetilde{S}_Bu_{\Delta},u_{\Delta}\rangle.
\end{align*}
Then we estimate the lower bound of eigenvalues of $P_{NN}^{-1}(\widetilde{S}_R+\widetilde{S}_B)$ instead of $P_{NN}^{-1}\widetilde{S}$. Since $\widetilde{S}_R, \widetilde{S}_B$ are both positive definite operators, it holds that
\begin{align*}
\langle P_{NN2}^{-1}(\widetilde{S}_R+\widetilde{S}_B)u_{\Delta},u_{\Delta}\rangle &= \left((\delta_R^{\dag})^2+(\delta_B^{\dag})^2\right)\langle u_{\Delta},u_{\Delta}\rangle+(\delta_R^{\dag})^2\langle\widetilde{S}_R^{-1}\widetilde{S}_Bu_{\Delta},u_{\Delta}\rangle\\
&+(\delta_B^{\dag})^2\langle\widetilde{S}_B^{-1}\widetilde{S}_Ru_{\Delta},u_{\Delta}\rangle\ge\left((\delta_R^{\dag})^2+(\delta_B^{\dag})^2\right)\langle u_{\Delta},u_{\Delta}\rangle,
\end{align*}
therefore,
\begin{equation}\label{manynnlower}
\left((\delta_R^{\dag})^2+(\delta_B^{\dag})^2\right)\langle u_{\Delta},u_{\Delta}\rangle\le\langle P_{NN}^{-1}\widetilde{S}u_{\Delta},u_{\Delta}\rangle.
\end{equation}
As $\delta_R^{\dag}+\delta_B^{\dag} = 1$, it holds that $(\delta_R^{\dag})^2+(\delta_B^{\dag})^2\ge \frac12$. The lower bound is obtained.\\
\indent
We then estimate the upper bound. By using (\ref{manydn2}), we have
\begin{equation}\label{manynn1}
\langle\widetilde{S}u_{\Delta},u_{\Delta}\rangle\le\left(1+C\frac{\nu_R}{\nu_B}(1+\log{\frac Hh})^2\right)\langle\widetilde{S}_Bu_{\Delta},u_{\Delta}\rangle
\end{equation}
\begin{equation}\label{manynn2}
\langle\widetilde{S}u_{\Delta},u_{\Delta}\rangle\le\left(1+C\frac{\nu_B}{\nu_R}(1+\log{\frac Hh})^2\right)\langle\widetilde{S}_Ru_{\Delta},u_{\Delta}\rangle.
\end{equation}
Combining (\ref{manynn1}), (\ref{manynn2}), we get
\begin{equation}
\lambda(P_{NN}^{-1}\widetilde{S})\le(\delta_R^{\dag})^2+(\delta_B^{\dag})^2+ C((\delta_R^{\dag})^2\frac{\nu_B}{\nu_R}+(\delta_B^{\dag})^2\frac{\nu_R}{\nu_B})(1+\log{\frac Hh})^2.
\end{equation}
Similar to the case of two subdomains, if we set $\delta_R^{\dag} = 1$ or $0$, the method becomes D-N algorithm actually. This is not a good choice. One of the optimal choices could be
\begin{equation*}
\delta_R^{\dag} = \frac{\sqrt{\nu_R}}{\sqrt{\nu_R}+\sqrt{\nu_B}},\quad\delta_B^{\dag} = \frac{\sqrt{\nu_B}}{\sqrt{\nu_R}+\sqrt{\nu_B}},
\end{equation*}
then the upper bound, independent of $\nu_R, \nu_B$ is obtained as follows,
\begin{equation}\label{manynnupper}
\lambda(P_{NN2}^{-1}\widetilde{S})\le C\frac{\nu_R+\nu_B}{(\sqrt{\nu_R}+\sqrt{\nu_B})^2}(1+\log{\frac Hh})^2\le C(1+\log{\frac Hh})^2.
\end{equation}\\
\indent
Combining (\ref{manynnlower}) and (\ref{manynnupper}), we get the conclusion (\ref{thmNN2}).
\qed
\end{proof}\\
\indent
The conclusion and the proof of D-D algorithm is similar to that of N-N algorithms.\\

\indent
We have analyzed four kinds of domain decompositions in the case of many subdomains. Similar to the case of two subdomains, we find that they have two kinds of different behaviours when applied in the discontinuous coefficients case. Why there is such a difference? Intuitively, the D-N algorithm and R-R algorithm use information of half the subdomains to precondition the whole system while energy norms of $\Omega_R$ and $\Omega_B$ are controlled by each other in the N-N algorithm and D-D algorithm. Therefore, D-N algorithm and R-R algorithm may perform very well in some discontinuous coefficients case as they fully take advantage of the ratio $\frac{\nu_R}{\nu_B}$ but N-N algorithm and D-D algorithm do not have such a good property, although their condition number bounds may be independent of the discontinuous coefficients by choosing appropriate weights and they are more applicable in complex coefficients cases.
\subsection{Implementation of the algorithms}
In the subsection, we will describe the implementation of the preconditioned systems and right hand sides.\\
\indent
We first illustrate the implementation of $\widetilde{S}, \widetilde{S}_R$ and $\widetilde{S}_B$. Reorder the vectors of unknowns into the following form:
\begin{equation*}
u_R^T =
\begin{pmatrix}
{u_I^R}^T & u_{\Delta}^T & u_C^T
\end{pmatrix}\quad
u_B^T =
\begin{pmatrix}
{u_I^B}^T & u_{\Delta}^T & u_C^T
\end{pmatrix}\quad
u^T =
\begin{pmatrix}
u_I^T & u_{\Delta}^T & u_C^T
\end{pmatrix},
\end{equation*}
then we have
\begin{equation*}
\begin{pmatrix}
A_{II} & A_{I\Delta} & A_{IC}\\
A_{\Delta I} & A_{\Delta\Delta} & A_{\Delta C}\\
A_{CI} & A_{C\Delta} & A_{CC}\\
\end{pmatrix}
\begin{pmatrix}
u_I \\ u_{\Delta} \\ u_C
\end{pmatrix} =
\begin{pmatrix}
f_I \\ f_{\Delta} \\ f_C
\end{pmatrix}.
\end{equation*}
The systems of $u_R, u_B$ are similar.
The Schur complement system on $\Gamma_{\Delta}$ is as follows:
\begin{equation*}
\widetilde{S}u_{\Delta} = \tilde{f}_{\Delta},
\end{equation*}
where
\begin{equation*}
\widetilde{S} = A_{\Delta\Delta}-
\begin{pmatrix}
A_{\Delta I} & A_{\Delta C}
\end{pmatrix}
\begin{pmatrix}
A_{II} & A_{IC}\\
A_{CI} & A_{CC}\\
\end{pmatrix}^{-1}
\begin{pmatrix}
A_{I\Delta} \\ A_{C\Delta}
\end{pmatrix},
\end{equation*}
and
\begin{equation*}
\tilde{f}_{\Delta} = f_{\Delta}-
\begin{pmatrix}
A_{\Delta I} & A_{\Delta C}
\end{pmatrix}
\begin{pmatrix}
A_{II} & A_{IC}\\
A_{CI} & A_{CC}\\
\end{pmatrix}^{-1}
\begin{pmatrix}
f_I \\ f_C
\end{pmatrix}.
\end{equation*}
Similarly, we may get $\widetilde{S}_R, \widetilde{S}_B$. In the following, we should know how $\widetilde{S}$ and $\widetilde{S}_R^{-1}, \widetilde{S}_B^{-1}$ act on a given vector $u_{\Delta}$. To determine $\widetilde{S}u_{\Delta}$, we first solve the following coarse problem
\begin{equation*}
S_{CC}u_C = \hat{f}_C,
\end{equation*}
where
\begin{equation*}
S_{CC} = A_{CC}-A_{CI}A_{II}^{-1}A_{IC},
\end{equation*}
and
\begin{equation*}
\tilde{f}_C = A_{C\Delta}u_{\Delta}-A_{CI}A_{II}^{-1}A_{I\Delta}u_{\Delta}.
\end{equation*}
Then we need to solve subdomain Dirichlet problems with boundary data given by $u_{\Delta}$ and $u_C$, and finally obtain $\widetilde{S}u_{\Delta}$. $\widetilde{S}_Ru_{\Delta}, \widetilde{S}_Bu_{\Delta}$ could be obtained in the same way. To compute $\widetilde{S}_R^{-1}u_{\Delta}$, we need to solve the following problem,
\begin{equation}\label{matinvS}
\begin{pmatrix}
A_{II}^R & A_{I\Delta}^R & A_{IC}^R\\
A_{\Delta I}^R & A_{\Delta\Delta}^R & A_{\Delta C}^R\\
A_{CI}^R & A_{C\Delta}^R & A_{CC}^R\\
\end{pmatrix}
\begin{pmatrix}
w_I \\ w_{\Delta} \\ w_C\\
\end{pmatrix} =
\begin{pmatrix}
0 \\ u_{\Delta} \\ 0
\end{pmatrix}.
\end{equation}
By eliminating the unknowns of type $I$ and $\Delta$, we get the coarse problem
\begin{equation*}
\widetilde{S}_{CC}^Rw_C = \tilde{f}_C^R,
\end{equation*}
where
\begin{equation*}
\widetilde{S}_{CC}^R = A_{CC}^R-
\begin{pmatrix}A_{CI}^R & A_{C\Delta}^R\end{pmatrix}
\begin{pmatrix}
A_{II}^R & A_{I\Delta}^R\\
A_{\Delta I}^R & A_{\Delta\Delta}^R\\
\end{pmatrix}^{-1}
\begin{pmatrix}
A_{IC}^R \\ A_{\Delta C}^R
\end{pmatrix},
\end{equation*}
and
\begin{equation*}
\tilde{f}_C^R = -
\begin{pmatrix}A_{CI}^R & A_{C\Delta}^R\end{pmatrix}
\begin{pmatrix}
A_{II}^R & A_{I\Delta}^R\\
A_{\Delta I}^R & A_{\Delta\Delta}^R\\
\end{pmatrix}^{-1}
\begin{pmatrix}
0 \\ u_{\Delta}
\end{pmatrix}.
\end{equation*}
After solving $w_C$, we may substitute it into (\ref{matinvS}) and solve local problems, then $w_{\Delta}$ is the desired vector $\widetilde{S}_R^{-1}u_{\Delta}$. The implementation of $\widetilde{S}_B^{-1}$ is similar.\\
\indent
Then we illustrate the implementation of $S_R$ and $S_B$. We take $S_R$ as an example. Reorder the vectors $u_R$ into the following form:
\[
u_R^T = \begin{pmatrix}{u_I^R}^T u_{\Gamma}^T \end{pmatrix}.
\]
Then we have
\begin{equation}
\begin{pmatrix}
A_{II}^R & A_{I\Gamma}^R\\
A_{\Gamma I}^R & A_{\Gamma\Gamma}^R
\end{pmatrix}
\begin{pmatrix}
u_I^R\\u_{\Gamma}
\end{pmatrix} =
\begin{pmatrix}
f_I^R\\f_{\Gamma}^R
\end{pmatrix},
\end{equation}
By eliminating the interior component, we have
\[S_Ru_{\Gamma} = f_R,\]
where
\[
S_R = A_{\Gamma\Gamma}^R-A_{\Gamma I}^R{A_{II}^R}^{-1}A_{I\Gamma}^R,
\]
and
\[
f_R = f_{\Gamma}^R-A_{\Gamma I}^R{A_{II}^R}^{-1}f_I^R.
\] $S_B, f_B$ are obtained in the same way.\\
\indent
The implementation of $S_R$ and $S_B$ acting on a given vector is realized by Gaussian block elimination. And the implementation of $S_R^{-1}, S_B^{-1}, (\gamma_RM+S_R)^{-1}, (\gamma_BM+S_B)^{-1}, (\gamma_RM-S_B)^{-1}$ acting on a vector may follow the same way of $\widetilde{S}_R^{-1}$ with the right hand side in the following form,\[g^T = \begin{pmatrix} 0 & g_{\Delta}^T & g_C^T\end{pmatrix}^T,\] as they act on vectors of $V_{\Gamma}$.\\
\indent
Now the only thing left is the right hand side of R-R system. In fact, the R-R algorithm should be treated carefully and the system after simplification is
\begin{equation*}
M\left((\gamma_RM-S_B)^{-1}-(\gamma_RM+S_R)^{-1})\right)Mg_R = M(\gamma_RM-S_B)^{-1}f_B+M(\gamma_RM+S_R)^{-1}f_R.
\end{equation*}
and $P_{RR}^{-1} = (\gamma_R+\gamma_B)(\gamma_BM+S_B)^{-1}-M^{-1}$, where $M$ is the mass matrix of $V_{\Gamma}$.\\
\indent
So far, the implementation of the algorithms is completed.

\section{Numerical experiments}
In this section, we perform some numerical experiments to verify our conclusions. We consider the following diffusion problem with zero Dirichlet boundary condition,
\begin{equation}
-\nabla\cdot(\nu(\mathbf{x})\nabla u) = f\quad {\rm in\ }\Omega,
\end{equation}
where $\Omega = (0,1)^2$ and $f = -2(x^2+y^2-x-y)$.\\
\indent
We first test the case of two subdomains with $\Omega_1, \Omega_2$ symmetric with respect to $\Gamma = \{\frac 12\}\times(0,1)$, i.e. $\Omega_1 = (0,\frac 12)\times(0,1), \Omega_2 = (\frac 12,1)\times(0,1)$. The iteration stops when the relative error is less than $tol = 10^{-8}$. The weights of N-N algorithm and D-D algorithms are set to be optimal and the Robin parameters of R-R algorithm are $\gamma_1 = \nu_2/h$ and $\gamma_2 = \nu_1$.\\
\begin{table}[H]
\caption{Numbers of iterations of D-N algorithm and N-N algorithm with different parameters}
  \centering\label{2sdnnn}
    \begin{tabular}{ccccccccc}
    \toprule
    \toprule
\multirow{2}{*}{$\nu_1$} &\multirow{2}{*}{$\nu_2$} &\multirow{2}{*}{$h$} & \multicolumn{3}{c}{D-N}&\multicolumn{3}{c}{N-N}\cr
\cmidrule(lr){4-6} \cmidrule(lr){7-9}
& & &$\theta_{opt}$& 1/2 & 1 &$\theta_{opt}$&1/3&2/3\cr
    \midrule
    $10^{-2}$  & $10^{2}$ & 1/16 & 1 & 27 & 3 & 1 & 18 & 16 \cr
    $10^{-4}$  & $10^{4}$ & 1/16 & 1 & 27 & 1 & 1 & 17 & 17 \cr
    $10^{-6}$  & $10^{6}$ & 1/16 & 1 & 27 & 1 & 1 & 17 & 17 \cr
    $10^{-2}$  & $10^{2}$ & 1/32 & 1 & 27 & 3 & 1 & 18 & 16 \cr
    $10^{-4}$  & $10^{4}$ & 1/32 & 1 & 27 & 1 & 1 & 17 & 17 \cr
    $10^{-6}$  & $10^{6}$ & 1/32 & 1 & 27 & 1 & 1 & 17 & 17 \cr
    $10^{-2}$  & $10^{2}$ & 1/64 & 1 & 27 & 3 & 1 & 18 & 16 \cr
    $10^{-4}$  & $10^{4}$ & 1/64 & 1 & 27 & 2 & 1 & 17 & 17 \cr
    $10^{-6}$  & $10^{6}$ & 1/64 & 1 & 27 & 1 & 1 & 17 & 17 \cr
    \bottomrule
    \bottomrule
    \end{tabular}
\end{table}

\begin{table}[H]
\caption{The numbers of iterations of D-D algorithm and R-R algorithm with different parameters}
  \centering\label{2sddrr}
    \begin{tabular}{ccccccccc}
    \toprule
    \toprule
\multirow{2}{*}{$\nu_1$} &\multirow{2}{*}{$\nu_2$} &\multirow{2}{*}{$h$} & \multicolumn{3}{c}{D-D}&\multicolumn{3}{c}{R-R}\cr
\cmidrule(lr){4-6} \cmidrule(lr){7-9}
& & &$\theta_{opt}$&1/3&2/3&$\theta_{opt}$&1/2&1\cr
    \midrule
    $10^{-2}$  & $10^{2}$ & 1/16 & 1 & 18 & 16 & 2 & 27 & 2 \cr
    $10^{-4}$  & $10^{4}$ & 1/16 & 1 & 17 & 17 & 1 & 27 & 1 \cr
    $10^{-6}$  & $10^{6}$ & 1/16 & 1 & 17 & 17 & 1 & 27 & 1 \cr
    $10^{-2}$  & $10^{2}$ & 1/32 & 1 & 18 & 16 & 2 & 27 & 2 \cr
    $10^{-4}$  & $10^{4}$ & 1/32 & 1 & 17 & 17 & 1 & 27 & 1 \cr
    $10^{-6}$  & $10^{6}$ & 1/32 & 1 & 17 & 17 & 1 & 27 & 1 \cr
    $10^{-2}$  & $10^{2}$ & 1/64 & 1 & 18 & 16 & 2 & 27 & 2 \cr
    $10^{-4}$  & $10^{4}$ & 1/64 & 1 & 17 & 17 & 1 & 27 & 1 \cr
    $10^{-6}$  & $10^{6}$ & 1/64 & 1 & 17 & 17 & 1 & 27 & 1 \cr
    \bottomrule
    \bottomrule
    \end{tabular}
\end{table}
Table \ref{2sdnnn} and Table \ref{2sddrr} show the numbers of iterations of different algorithms with several groups of parameters. $\theta_{opt}$ denotes the optimal $\theta$ mentioned above. We may find that if we choose the optimal $\theta$, D-N algorithm, N-N algorithm and D-D algorithm converge in one step which corresponds to the zero convergence rate in Theorem \ref{0convergence}. For D-N algorithm and R-R algorithm, if $\theta$ is set to be 1, the convergence rate will thoroughly rely on the ratio $\nu_1/\nu_2$ and the data in Table \ref{2sdnnn} and Table \ref{2sddrr} confirm the conclusion because the iteration counts decrease with the jump in the coefficient increases. For N-N algorithm and D-D algorithm, we find that the iterations are not affected by the discontinuous coefficients $\nu_1, \nu_2$ which supports the theoretical results. Besides, we note that all of them are independent of mesh size $h$.\\
\\
\indent
In the second experiment, we test the case of two subdomains with $\Omega_1, \Omega_2$ nonsymmetric. The first interface $\Gamma_1$ is set to be $\{\frac 14\}\times(0,1)$ and the second case is $\Gamma_2 = \{\frac 34\}\times(0,1)$. The mesh size $h$ is fixed to be $1/64$. The relaxation parameter is always the optimal one of symmetric case. The other settings are the same as the previous experiment.\\
\begin{table}[H]
\caption{The numbers of iterations with different discontinuous coefficients}
  \centering\label{2ns}
    \begin{tabular}{cccccccccc}
    \toprule
    \toprule
\multirow{2}{*}{$\nu_1$} &\multirow{2}{*}{$\nu_2$} & \multicolumn{2}{c}{D-N} & \multicolumn{2}{c}{N-N} & \multicolumn{2}{c}{D-D} & \multicolumn{2}{c}{R-R}\cr
\cmidrule(lr){3-4} \cmidrule(lr){5-6} \cmidrule(lr){7-8} \cmidrule(lr){9-10}
& &$\Gamma_1$&$\Gamma_2$&$\Gamma_1$&$\Gamma_2$&$\Gamma_1$&$\Gamma_2$&$\Gamma_1$&$\Gamma_2$\cr
    \midrule
    $10^{-1}$  & $10^{1}$ & 4 & 4 & 11 & 14 & 11 & 14 & 5 & 5\cr
    $10^{-2}$  & $10^{2}$ & 2 & 2 & 11 & 14 & 11 & 14 & 2 & 2\cr
    $10^{-3}$  & $10^{3}$ & 2 & 2 & 11 & 14 & 11 & 14 & 2 & 2\cr
    $10^{-4}$  & $10^{4}$ & 1 & 1 & 11 & 14 & 11 & 14 & 1 & 1\cr
    $10^{-5}$  & $10^{5}$ & 1 & 1 & 11 & 14 & 11 & 14 & 1 & 1\cr
    $10^{-6}$  & $10^{6}$ & 1 & 1 & 11 & 14 & 11 & 14 & 1 & 1\cr
    \bottomrule
    \bottomrule
    \end{tabular}
\end{table}
\indent
In Table \ref{2ns}, the convergence of the different methods is shown for different coefficient ratios in nonsymmetric case. The D-N algorithm and R-R algorithm with optimal $\theta$ converge very quickly. This is because their convergence rates are dominated by the coefficient ratio and the smaller the ratio is, the faster the iteration converges. For the N-N algorithm and D-D algorithm with optimal weights, the convergence rates are independent of the discontinuous coefficients. In addition, by comparing the numbers of iteration with $\Gamma_1, \Gamma_2$, we may see that the influence of using different interfaces is little.\\
\\
\indent
At last, we test the case of many subdomains with coefficients red-black checkerboard distribution. We use the PCG method and the terminal precision is chosen as $10^{-6}$. The weights of N-N algorithm and D-D algorithm are optimal and the Robin parameters $\gamma_R, \gamma_B$ are set to be $16\nu_B/h$ and $\nu_RH/2$, respectively.
\begin{table}[H]
\caption{The numbers of iteration for $8\times 8$ subdomains with $\nu_R = \nu_B = 1$}
  \centering\label{manyHh}
    \begin{tabular}{ccccc}
    \toprule
    \toprule
$\frac Hh$& D-N & N-N & D-D & R-R \cr
    \midrule
4 & 15 & 8  & 7  & 15 \cr
8 & 17 & 10 & 8  & 17 \cr
16& 19 & 11 & 9  & 19 \cr
32& 21 & 13 & 10 & 21 \cr
64& 23 & 14 & 11 & 23 \cr
    \bottomrule
    \bottomrule
    \end{tabular}
\end{table}

\begin{table}[H]
\caption{The numbers of iteration for case of many subdomains with $\nu_R = \nu_B = 1$ and fixed $\frac Hh = 8$}
  \centering\label{manyNN}
    \begin{tabular}{ccccc}
    \toprule
    \toprule
$N\times N$& D-N & N-N & D-D & R-R \cr
    \midrule
$4\times4$  & 9  & 5  & 4 & 10 \cr
$8\times8$  & 17 & 10 & 8 & 17 \cr
$16\times16$& 20 & 10 & 8 & 20 \cr
$24\times24$& 20 & 10 & 8 & 20 \cr
$32\times32$& 20 & 10 & 7 & 20 \cr
    \bottomrule
    \bottomrule
    \end{tabular}
\end{table}
\indent
Table \ref{manyHh} shows the corresponding iteration numbers of the four algorithms when the mesh is refined. The slight increases of the iteration numbers explain that the condition number is growing slowly as $\frac Hh$ increases. Then we fix the degrees of freedom in each subdomain, that is, $\frac Hh$ is a constant, we get the results in Table $\ref{manyNN}$. We may see that the iteration number of each algorithm is stable which reflects that the iteration numbers rely only on $\frac Hh$.

\begin{table}[H]
\caption{The numbers of iteration for $8\times 8$ subdomains with different discontinuous coefficients}
  \centering\label{manydc}
    \begin{tabular}{cccccccc}
    \toprule
    \toprule
$\nu_B$& $\nu_R$ & D-N & N-N & D-D & R-R \cr
    \midrule
$10^{1}$ & $10^{-1}$ & 4 &  17  & 14  & 4 \cr
$10^{2}$ & $10^{-2}$ & 2 &  17  & 14  & 2 \cr
$10^{3}$ & $10^{-3}$ & 2 &  17  & 14  & 2 \cr
$10^{4}$ & $10^{-4}$ & 1 &  17  & 14  & 1 \cr
$10^{5}$ & $10^{-5}$ & 1 &  17  & 14  & 1 \cr
$10^{6}$ & $10^{-6}$ & 1 &  17  & 14  & 1 \cr
    \bottomrule
    \bottomrule
    \end{tabular}
\end{table}
\indent
The results in table \ref{manydc} corresponds to the numerical experiment with fixed numbers of subdomains, fixed $\frac Hh = 8$ and increasing jumps in the coefficients. We may see that the iteration numbers of D-N algorithms and R-R algorithm decrease rapidly as the jumps increase which confirms our theory and the iteration numbers of N-N algorithms and D-D algorithm are stable.
\bibliographystyle{plain}
\bibliography{reference}
\end{document}